\documentclass[12pt]{amsart}
\usepackage{amssymb}
\newtheorem{theorem}{Theorem}[section]
\newtheorem{lemma}[theorem]{Lemma}
\newtheorem{proposition}[theorem]{Proposition}
\newtheorem{corollary}[theorem]{Corollary}

\theoremstyle{definition}

\newtheorem{examp}[theorem]{Example}

\theoremstyle{remark}

\numberwithin{equation}{section}

\newcommand{\diams}{\unskip\nobreak\hfil\penalty50%
\hskip1em\hbox{}\nobreak\hfil%
$\diamondsuit$\parfillskip=0pt\finalhyphendemerits=0}

\newcommand{\bfind}[1]{\index{#1}{\bf #1}}
\newcommand{\n}{\par\noindent}

\newcommand{\sn}{\par\smallskip\noindent}

\newcommand{\bn}{\par\bigskip\noindent}
\newcommand{\pars}{\par\smallskip}
\newcommand{\parm}{\par\medskip}

\newcommand{\cO}{\mathcal{O}}

\newcommand{\cE}{\mathcal{E}}
\newcommand{\td}{d} 

\newcommand{\isom}{\simeq}

\newcommand{\ac}{^{\rm ac}}
\newcommand{\chara}{\mbox{\rm char}\,}
\newcommand{\Gal}{\mbox{\rm Gal}\,}
\newcommand{\dist}{\mbox{\rm dist}\,}
\newcommand{\N}{\mathbb N}

\newcommand{\R}{\mathbb R}
\newcommand{\Z}{\mathbb Z}

\newcommand{\F}{\mathbb F}

\begin{document}

\title[Finitely many defect extensions of prime degree]{Valued fields with finitely many defect
extensions of prime degree}

\author{F.-V.~Kuhlmann}
\address{Institute of Mathematics, University of Szczecin, ul.~Wielkopolska 15
70-451 Szczecin, Poland}
\email{fvk@usz.edu.pl}
\thanks{The author is supported by Opus grant 2017/25/B/ST1/01815 from the National Science Centre of Poland.
This paper was started while the author was on a research visit to the Department of Mathematics at the
Federal University of S\~{a}o Carlos, supported by the S\~{a}o Paulo Research Foundation (FAPESP), grant number
2017/17835-9. The author would like to thank the members of the department for their hospitality, and
Josnei Novacoski for drawing his attention to the gap in an earlier paper that is now filled at the end of this
paper. He also thanks the referee for the careful reading of the manuscript and many helpful corrections and
comments.}


\date{October 1, 2020}

\begin{abstract}
We prove that a valued field of positive characteristic $p$ that has only finitely many distinct Artin-Schreier
extensions (which is a property of infinite NTP$_2$ fields) is dense in its perfect hull. As a consequence, it is
a deeply ramified field and has $p$-divisible value group and perfect residue field. Further, we prove a partial
analogue for valued fields of mixed characteristic and observe an open problem about 1-units in this setting.
Finally, we fill a gap that occurred in a proof in an earlier paper in
which we first introduced a classification of Artin-Schreier defect extensions.
\end{abstract}

\maketitle

%
%
\section{Introduction}
It is shown in \cite[Theorem 3.1]{CKS} that an infinite field of positive characteristic that is definable
in an NTP$_2$ theory has only finitely many Artin-Schreier extensions. NTP$_2$ is a large class of first order
theories (``without the tree property of the second kind'') defined by S.~Shelah generalizing simple and NIP
theories. Algebraic examples of NTP$_2$ structures are given by ultraproducts of $p$-adic fields and certain
valued difference fields. For the precise definition (which we will not need in this paper) and for further
examples, we refer the reader to \cite{CKS}.

In \cite[Proposition~3.2]{CKS} it is shown that the value group of a valued field of characteristic $p>0$
which has only finitely many Artin-Schreier extensions is $p$-divisible. In this note, assuming throughout that
the valuation is nontrivial, we prove a stronger result,
namely, that such a field is dense in its perfect hull. For a valued field of characteristic $p>0$,
this is equivalent to being a deeply ramified field in the sense of \cite{GR}. The density also
implies that the value group is $p$-divisible and the residue field is perfect. By
Theorem~\ref{charst} in Section~\ref{sectdr},
every deeply ramified field of positive characteristic is a semitame field, which we shall define now.

Take a valued field $(K,v)$ of arbitrary characteristic. Its value group will be denoted by $vK$, and
its residue field by $Kv$. Accordingly, the value of an element $a\in K$ will be denoted by $va$, and its residue
by $av$. We say that $(K,v)$ is \bfind{semitame} if either $\chara Kv=0$, or $\chara Kv=p>0$ and the following two
axioms hold:
\sn
{\bf (DRst)} the value group $vK$ is $p$-divisible,
\sn
{\bf (DRvr)} the homomorphism
\[
\cO_{K^c}/p\cO_{K^c} \ni x\mapsto x^p\in \cO_{K^c}/p\cO_{K^c}
\]
is surjective, where $\cO_{K^c}$ denotes the valuation ring of the completion $K^c$ of $(K,v)$.
Note that this condition implies that the residue field $Kv$ is perfect.

\pars
The following result will be proven in Section~\ref{sectcharppf}:
\begin{theorem}                  \label{MTfinAS}
Take a valued field $(K,v)$ of characteristic $p>0$. If $K$ admits only finitely many distinct Artin-Schreier
extensions, then $(K,v)$ is a semitame field.
\end{theorem}

\begin{corollary}
A nontrivially valued field of positive characteristic that is definable in an NTP$_2$ theory is a semitame field.
\end{corollary}

We will prove Theorem~\ref{MTfinAS} by showing that if $(K,v)$ is not dense in its perfect hull, then it admits
infinitely many distinct Artin-Schreier extensions; see Proposition~\ref{propnotdense}.
However, we will show more than this. We are interested in Galois \bfind{defect extensions} of degree $p$ a prime. 
These are immediate Galois extensions $(L|K,v)$ of degree $p$ of valued fields for which $v$ has a unique 
extension from $K$ to $L$ (in this case we must have that $p=\chara Kv$). Here, $(L|K,v)$
denotes an extension $L|K$ of fields with $v$ a valuation on $L$ and $K$ endowed with its restriction; the
extension is said to be \bfind{immediate} if the canonical embeddings of $vK$ in $vL$ and of $Kv$ in $Lv$ are
onto. For more details on the defect see Section~\ref{sectfi}, and for further background, see
\cite{BKdrf,Ku1,Ku6}.

In \cite{BKdrf} a classification of Galois defect extensions $\cE=(L|K,v)$ of prime degree $p$ is given as follows.
We show that the set
\[
\Sigma_\sigma\>:=\> \left\{ v\left( \left.\frac{\sigma f-f}{f}\right) \right| \, f\in L^{\times} \right\} 
\]
is independent of the choice of a generator $\sigma$ of $\Gal (L|K)$, and we
denote it by $\Sigma_\cE\,$. We say that $\cE$ has \bfind{independent defect} if
\[
\Sigma_{\cE}\>=\> \{\alpha\in vK\mid \alpha >H_\cE\} 
\]
for some proper convex subgroup $H_\cE$ of $vK$;  otherwise we say that $\cE$ has \bfind{dependent defect}. If 
$vK$ is archimedean ordered, or in other words, $(K,v)$ is of \bfind{rank $1$}, then $H_\cE$ can only be equal
to $\{0\}$.
\pars
In Section~\ref{sectcharppf}, we will prove:
\begin{theorem}                                \label{MTdepAS}
Take a valued field $(K,v)$ of characteristic $p>0$. If it admits an Artin-Schreier extension with dependent
defect, then it admits infinitely many distinct Artin-Schreier extensions with dependent defect.
\end{theorem}

In \cite{Ku6} the classification was originally introduced only for valued fields of positive characteristic.
An Artin-Schreier defect extension $(L|K,v)$ was said to have dependent defect if in a certain way it 
depends on immediate purely inseparable extensions of degree $p$ which do not lie in the completion of $(K,v)$.
It was then shown in Section~4.2 of \cite{Ku6} that an Artin-Schreier extension has dependent defect if it is 
obtained from such a purely inseparable defect extension by a certain transformation of an inseparable minimal 
polynomial which makes it separable. We will describe this transformation in Section~\ref{secttransf} and use it 
for the proof of Theorem~\ref{MTdepAS}.

In \cite{BKdrf} we have made the above more precise by showing that {\it every} Artin-Schreier extension with 
dependent defect can be obtained by such a transformation. It is also shown that the new definition of the 
classification given in \cite{BKdrf} is compatible with the one given in \cite{Ku6}. 

\pars
The new definition of the classification became necessary in order to generalize the original definition to the 
case of valued fields $(K,v)$ of characteristic $0$ with residue field $Kv$ of positive characteristic $p$ 
(\bfind{mixed characteristic}), where we cannot rely on nontrivial purely inseparable extensions. We will now 
present a partial analogue of Theorem~\ref{MTdepAS} in the mixed characteristic case. To
avoid a number of technical details in the present paper, we will restrict our scope to valued fields $(K,v)$  
of rank $1$. We will also assume that $K$ contains a primitive $p$-th root of
unity. We then consider Kummer defect extensions of degree $p$, that is, Galois defect extensions of $(K,v)$
generated by some $\eta\notin K$ such that $\eta^p\in K$. Because such extensions are immediate,
we can show that we can assume $\eta$ to be a \bfind{$1$-unit}, i.e., $v(\eta-1)>0$ (and thus, $v\eta=0$); see 
Section~\ref{sectK}.

We need some more preparation. Take an extension $(L|K,v)$ and $a\in L$. We define:
\[
v(a-K)\>:=\>\{v(a-c)\mid c\in K\}\>;
\]
for details on this set, see Section~\ref{sectva-K}.
For $\alpha$ in the divisible hull $\widetilde{vK}$ of the value group $vK$, we will write $v(a-K)<\alpha$
if $v(a-c)<\alpha$ for all $c\in K$, and similarly for ``$\leq$'' in place of ``$<$''. We write
\[
v(a-K)<\!\!|\,\alpha
\]
if $v(a-K)<\alpha$ and $v(a-K)$ is bounded away from $\alpha$ in  $\widetilde{vK}$, i.e.,
there is $\beta\in \widetilde{vK}$ such that $v(a-K)\leq\beta<\alpha$.

A Kummer defect extension of degree $p$ generated by a $1$-unit $\eta$ as above always satisfies
\begin{equation}                             \label{Kde}
v(\eta-K)\><\>\frac{vp}{p-1}\>,
\end{equation}
see Section~\ref{sectK}. In the case of archimedean value groups (where $H_\cE$ cn only be equal to $\{0\}$), the
definition we have given for the extension to have dependent defect is equivalent to the condition that
$v(\eta-K)<\!\!|\, \frac{vp}{p-1}$ (see our discussion at the end of Section~\ref{sectK}).

Note that in \cite{BKdrf} in place of $v(\eta-K)$ a more flexible invariant $\dist(\eta,K)$ is used for our work
with defect extensions. We will give its definition in Section~\ref{sectdist} to enable the reader to relate the results in \cite{BKdrf} and \cite{Ku6} to those in the present paper.

\pars
In view of Theorem~\ref{MTdepAS} it seems reasonable to conjecture that if $(K,v)$ admits a Kummer extension
with dependent defect, then it admits infinitely many distinct Kummer extensions with dependent defect. By
\cite[part 2) of Theorem 1.5]{BKdrf}, this would imply that if $(K,v)$ admits only finitely many Kummer defect
extensions, then it is semitame, in perfect analogy to the equal positive characteristic case. However, so far we
have only been able to prove a weaker result. We will say that a Kummer defect extension as above has
\bfind{super-dependent defect} if
\begin{equation}                                   \label{sd}
v(\eta-K)<\!\!|\, \frac{vp}{p}\>.
\end{equation}
\begin{theorem}                                      \label{MTmixed}
Take a valued field $(K,v)$ of mixed characteristic and rank 1 which contains a primitive $p$-th root of
unity. If it admits a Kummer extension of degree $p$ with super-dependent defect, then it admits infinitely
many distinct Kummer extensions of degree $p$ with super-dependent defect.
\end{theorem}
\n
This leads us to the following
\n
{\bf Open questions:} Does the above theorem also hold with dependent defect in place of super-dependent defect?
What is (possibly) special about a Kummer extension with dependent defect that is not super-dependent?
\sn
We have observed in earlier work already that the threshold $\frac{vp}{p}$ plays a certain role when dealing
with 1-units in mixed characteristic (see
\cite[Corollary 2.11 d)]{Ku8}), but it is not yet sufficiently understood what exactly this role is.

\parm
In the last section of this paper we extend our study of Artin-Schreier extensions of valued fields in order to
fill a gap that occurred in the proof of Lemma 2.31 of \cite{Ku6}, which gives a criterion for such extensions to
have nontrivial defect.

\parm
For general background on valuation theory, we refer the reader to \cite{[En],[EP],[R],[ZS2]}.
%
%
\section{Preliminaries}                              \label{sectprel}
%

%
%
\subsection{Deeply ramified and semitame fields}                              \label{sectdr}
For a field $K$ of characteristic $p>0$ we denote by $K^{1/p^\infty}$ the perfect hull
of $K$. Further, we set $K^p=\{a^p\mid a\in K\}$ and $K^{1/p}=\{a^{1/p}\mid a\in K\}$; then $K^p$ is a subfield
of $K$, and $K^{1/p}|K$ is a field extension which is trivial if and only if $K$ is perfect.

\pars
In \cite{BKdrf}, the following is proven:
\begin{theorem}                          \label{charst}
Take a nontrivially valued field $(K,v)$ of characteristic $p>0$. Then the following statements are equivalent:
\n
a) $(K,v)$ is a semitame field,
\n
b) $(K,v)$ is a deeply ramified field,
\n
c) $(K,v)$ satisfies (DRvr),
\n
d) the completion of $(K,v)$ is perfect,
\n
e) $(K,v)$ is dense in its perfect hull, i.e., $K^{1/p^\infty}\subset K^c$,
\n
f) $K^p$ is dense in $(K,v)$.
\end{theorem}
Here, the property ``dense'' refers to the topology induced by the valuation.
\pars
See \cite{BKdrf} for more details and the connection of semitame and deeply ramified fields with the
classification of defect extensions.

%
%
\subsection{The set $v(a-K)$}                              \label{sectva-K}
Take a totally ordered set $(T,<)$. For a nonempty subset $S$ of $T$ and an element $t\in T$ we will write $S<t$ if
$s<t$ for every $s\in S$. A set $S\subseteq T$ is called an \bfind{initial segment} of $T$ if for each
$s\in S$ every $t <s$ also lies in $S$. Similarly, $S\subseteq T$ is called
a \bfind{final segment} of $T$ if for each $s\in S$ every $t >s$ also lies in $S$.

If $(T,<)$ is an ordered abelian group, $n\in\N$, $\alpha\in T$, and $S$ is an initial segment of $T$, then
\[
nS+\alpha\>:=\>\{n\beta+\alpha\mid \beta\in S\}
\]
is an initial segment of $nT$.

\pars
Take a valued field extension $(K(a)|K,v)$.  We will now collect various properties of the sets $v(a-K)$ and
$v(a-K)\cap vK$.
\begin{proposition}                          \label{propva-K}
Take a valued field extension $(K(a)|K,v)$.
\sn
1) If $(K(a)|K,v)$ is immediate, then $v(a-K)$ has no largest element.
\sn
2) If $v(a-K)$ has no largest element, then $v(a-K)\subseteq vK$.
\sn
3) The set $v(a-K)\cap vK$ is an initial segment of $vK$.
\sn
4) The set $v(a-K)\setminus (v(a-K)\cap vK)$ has at most one element.
\sn
5) For every $c,d\in K$,
\[
v(da+c-K) \>=\> v(a-K)+vd\>.
\]
\sn
6) If $a,b$ are elements in some valued field extension of $(K,v)$ such that $v(a-b)>v(a-K)$, then
$v(a-c)=v(b-c)$ for all $c\in K$ and $v(a-K)=v(b-K)$.
\end{proposition}
\begin{proof}
1): This follows from \cite[Theorem 1]{[Ka]}.
\sn
2): Take $c\in K$; we wish to show that $v(a-c)\in vK$. By assumption there is $d\in K$ such that $v(a-d)>v(a-c)$.
Hence $v(a-c)=\min\{v(a-c),v(a-d)\}=v(c-d)\in vK$.
\sn
3): Take $\alpha\in v(a-K)\cap vK$ and $\beta\in vK$ such that $\beta<\alpha$. Choose $b,c\in K$ with $vb=\beta$
and $v(a-c)=\alpha$. Then $c-b\in K$ and $vb=\min\{v(a-c),vb\}=v(a-(c-b))\in v(a-K)$.
\sn
4): Assume that $v(a-c),v(a-d)\in v(a-K)\setminus (v(a-K)\cap vK)$. If they were distinct, say $v(a-d)>v(a-c)$,
then as in the proof of 2) we would obtain that $v(a-c)\in vK$, contradiction.
\sn
5): This follows from the equalities $v(a+c-K)=v(a-K)$ and $v(da-K)=v(a-K)+vd$, which are straightforward to prove.
\sn
6): For all $c\in K$, from $v(a-b)>v(a-c)$ we obtain that $v(a-c)=v(b-c)$ as in the proof of 2); hence $v(a-c)\in
v(b-K)$ and $v(b-c)\in v(a-K)$, which shows that $v(a-K)=v(b-K)$.
\end{proof}

We will mostly work with immediate extensions, in which case we have that $v(a-K)\subseteq vK(a)=vK$ for every element $a\notin K$ in the extension.
However, several of our results hold more generally for extensions that are not necessarily immediate. Some of them, such as the following one, remain true when the set $v(a-K)$ is replaced by $v(a-K)\cap vK$.
\begin{lemma}                               \label{distinct}
Take a valued field extension $(K(a)|K,v)$. Take $\alpha\in v(a-K)$ and assume that $d\in K$ is such that
\begin{equation}               \label{alph+vd}
\alpha+vd\> >\> v(a-K)
\end{equation}
(note that $d$ always exists when $v(a-K)$ is bounded by some element from $vK$). Then the sets
\[
v(a-K)-vd^n\>, \;n\in\N\>,
\]
are pairwise distinct.
\end{lemma}
\begin{proof}
Since $\alpha\in v(a-K)$ and $\alpha+vd>v(a-K)$, we know that $vd>0$. Take any natural numbers
$m<n$. Then
\[
\alpha-mvd\> \geq \> \alpha +vd-nvd \> >\> v(a-K)-nvd\>,
\]
which shows that $\alpha-mvd\notin v(a-K)-nvd=v(a-K)-vd^n$. But $\alpha-mvd\in
v(a-K)-mvd=v(a-K)-vd^m$, so the two sets are distinct.
\end{proof}

%
%
\subsection{Distances}                              \label{sectdist}
Take again a totally ordered set $(T,<)$.
A pair $(\Lambda^L,\Lambda^R)$  of subsets of $T$ is called a \bfind{cut} in $T$ if $\Lambda^L$ is an initial
segment of $T$ and $\Lambda^R=T\setminus \Lambda^L$; it then follows that $\Lambda^R$ is a final segment of $T$.
To compare cuts in  $(T,<)$ we will use the lower cut sets comparison. That is, for two cuts
$\Lambda_1=(\Lambda_1^L,\Lambda_1^R) ,\, \Lambda_2=(\Lambda_2^L,\Lambda_2^R)$ in $T$ we will write
$\Lambda_1<\Lambda_2$ if $\Lambda^L_1\varsubsetneq\Lambda^L_2$, and $\Lambda_1\leq \Lambda_2$ if
$\Lambda^L_1 \subseteq  \Lambda^L_2$. This defines a linear order on the set of all cuts in $T$.

For a given subset $S$ of $T$ we define $S^+$ to be the cut $(\Lambda^L,\Lambda^R)$ in $T$ such that
$\Lambda^L$ is the least initial segment containing $S$, that is,
\[
S^+\>:=\> (\{t\in T\,|\, \exists\, s\in S:\,t\leq s\}\,,\, \{t\in T\,|\, t>S\} )\>.
\]
Likewise, we denote by $S^-$ the cut $(\Lambda^L,\Lambda^R)$ in $T$ such that $\Lambda^L$ is the largest initial
segment disjoint from $S$, i.e.,
\[
S^-\>:=\> (\{t\in T\,|\, t<S\}\,,\, \{t\in T\,|\, \exists\, s\in S:\,t\geq s\} )\>.
\]
For $s\in T$, we set
\[
s^+\>:=\> \{s\}^+\quad\mbox{ and }\quad s^-\>:=\> \{s\}^-\>.
\]
We note that
\begin{equation}             \label{S+<s-}
S^+\>\leq\>s^- \>\Leftrightarrow\> S<s \;\mbox{ and }\;
S^+\><\>s^- \>\Leftrightarrow\> \exists\, t\in T:\, S\leq t<s\>.
\end{equation}
Indeed, $S^+\leq s^-$ means that the smallest initial segment of $T$ containing $S$ is a subset of the initial
segment $\{t\in T\,|\, t<s\}$, and this is equivalent to $S<s$. Likewise, $S^+< s^-$ means that the smallest
initial segment of $T$ containing $S$ is a proper
subset of $\{t\in T\,|\, t<s\}$, and as both are initial segments, this is equivalent to the existence of some
$t\in T$ such that $S\leq t$.
\pars
We embed $T$ in the linearly ordered set of all cuts in $T$ by identifying each $s\in T$ with $s^-$.

\pars
The set $v(a-K)\cap vK$ is an initial segment of $vK$ and thus the lower cut set of a cut in $vK$. However,
in order to be able to compare $v(a-K)$ with $v(a-L)$ when $L|K$ is algebraic, it is more convenient to work
with the cut $v(a-K)$ induces in the divisible hull $\widetilde{vK}$ of $vK$. Indeed,
$\widetilde{vK}$ is equal to the value group of the algebraic closure $K\ac$ of $K$, and if $L|K$ is algebraic,
then $K\ac=L\ac$ and therefore, $\widetilde{vK}=\widetilde{vL}$. Hence we define:
\[
\dist (a,K)\>:=\>(v(a-K)\cap \widetilde{vK})^+\;\; \textrm{ in the divisible hull $\widetilde{vK}$ of $vK\>$.}
\]
We call this cut the \bfind{distance of $z$ from $K$}. The distance replaces the use of suprema in the case of
non-archimedean value groups, which are not contained in $\R$.

\begin{lemma}                          \label{translate}
Let the situation be as above and assume that $v(a-K)\subseteq \widetilde{vK}$. Take some $\alpha\in
\widetilde{vK}$. Then $\dist (a,K)\leq\alpha^-$ is equivalent to $v(a-K)<\alpha$, and
$\dist (a,K)<\alpha^-$ is equivalent to $v(a-K)<\!\!|\, \alpha$.
\end{lemma}
\begin{proof}
We have:
\begin{eqnarray*}
\dist (a,K)\>\leq\>\alpha^-&\Leftrightarrow& v(a-K)\cap \widetilde{vK}<\alpha \\
&\Leftrightarrow& v(a-K)<\alpha\>,
\end{eqnarray*}
where the first equivalence holds by (\ref{S+<s-}) and the second equivalence holds since $v(a-K)\subseteq \widetilde{vK}$. Similarly,
\begin{eqnarray*}
\dist (a,K)\><\>\alpha^-&\Leftrightarrow& \exists\,\beta\in\widetilde{vK}:\, v(a-K)\cap
\widetilde{vK}\leq\beta<\alpha \\
&\Leftrightarrow& \exists\,\beta\in\widetilde{vK}:\, v(a-K)\leq\beta<\alpha
\>\Leftrightarrow\> v(a-K)<\!\!|\, \alpha\>.
\end{eqnarray*}
\end{proof}

\bn
%
%
\subsection{The fundamental equality and the defect}                              \label{sectfi}
Take an extension $(L|K,v)$ of valued fields such that the extension of $v$ from $K$ to $L$ is unique. Assume that
$\chara Kv=p>0$. Then
\begin{equation}                    \label{feuniq}
[L:K]\>=\> p^{\nu }\cdot(vL:vK)[Lv:Kv]\>,
\end{equation}
where by the Lemma of Ostrowski $\nu$ is a nonnegative integer (see \cite[Th\'eor\`eme 2, p.~236]{[R]}
or \cite[Corollary to Theorem 25, p.~78]{[ZS2]}). The factor
$d(L|K,v)=p^{\nu }$ is called the \bfind{defect} of the extension $(L|K,v)$.

\bn
%
%
\subsection{Criteria for defect extensions}                              \label{sectde}
%
%
\begin{lemma}                         \label{ueGp1}
Take an extension $(K(a)|K,v)$ of valued fields of degree
$p=\mbox{\rm char}(Kv)$ and such that the extension of $v$ from $K$ to $K(a)$ is unique.
If $v(a-K)$ has no maximal element, then $(K(a)|K,v)$ is immediate with defect~$p$.
\end{lemma}
\begin{proof}
By \cite[part (1) of Theorem 2.21]{Ku6}, $(K(a)|K,v)$ is immediate. As the extension of $v$ from $K$ to $K(a)$
is assumed to be unique, it follows that the defect is equal to the degree of the extension.
\end{proof}

\pars
%

\begin{lemma}	                              	\label{c1}\label{c2}
Assume that $(K(a)|K,v)$ is a normal extension of prime degree $p$.
\sn
1) If $K^h$ is some henselization of $(K,v)$ and $a\notin K^h$, then the extension of $v$ from $K$ to $K(a)$ is
unique.
\sn
2) Assume that $(K,v)$ is of rank 1. If $v(a-K)$ has no maximal element and is bounded in $vK$, then $(K(a)|K,v)$
is immediate and has defect $p$.
\end{lemma}
\begin{proof}
1): Since $K(a)|K$ is normal and of degree $p$, it is linearly disjoint from every other algebraic extension $L$
over $K$ in which it is not contained. This is seen as follows. If $K(a)|K$ is inseparable, then it is purely
inseparable, and so is $L(a)|L$; hence the latter extension can only be of degree $1$ or $p$. If $K(a)|K$ is
separable, then it is Galois with a Galois group of order $p$. Then also $L(a)|L$ is Galois. As restriction embeds
its Galois group in that of $K(a)|K$, again the degree of $L(a)|L$ can only be $1$ or $p$.

Since $a\notin K^h$ by assumption, from what we have just shown we see that $K(a)$ must be linearly disjoint from $K^h$ over $K$, which by
\cite[Lemma~2.1]{BK1} implies our assertion.
\sn
2): Since $(K,v)$ is of rank 1, $K$ lies dense in its henselization $K^h$. By assumption, $v(a-K)$ is bounded in
$vK$, which implies that $a\notin K^h$ and thus by part 1), the extension of $v$ from $K$ to $K(a)$ is unique.
Now our assertion follows from Lemma~\ref{ueGp1}.
\end{proof}

\bn
%
%
\subsection{Artin-Schreier extensions of valued fields}                              \label{sectAS}
Throughout this section, we let $K$ be a field of characteristic $p>0$.
Recall that $L|K$ is called an \bfind{Artin-Schreier
extension} if it is generated by a root of a polynomial of the form $X^p-X-b$ with $b\in K$. Note that every
Artin-Schreier extension of $K$ is a Galois extension of degree $p$, and vice versa.
We call $\vartheta$ an \bfind{Artin-Schreier generator} of an Artin--Schreier extension $L|K$ if $L=K(\vartheta)$
with $\vartheta^p-\vartheta\in K$.
\begin{lemma}                               \label{allASg}
Assume that $\vartheta$ is an Artin-Schreier generator of an Artin--Schreier extension $L|K$.
Then $\vartheta'$ is another Artin-Schreier generator
of $L|K$ if and only if $\vartheta' =i\vartheta +c$ for some $i\in\F_p$ and $c\in K$.
\end{lemma}
\begin{proof}
If $\vartheta$ and $\vartheta'$ are roots of the same polynomial $X^p-X-b$,
then $\vartheta-\vartheta'$ is a root of $X^p-X$, whose roots are exactly the elements of
$\F_p\,$. Hence, $\vartheta+i$, $i\in\F_p\,$, are all roots of $X^p-X-b$. Pick a nontrivial $\sigma\in
\Gal L|K$. We then have that $\sigma\vartheta-\vartheta=j$ for some $j\in\F_p^\times$.

If $\vartheta'$ is another Artin-Schreier generator of $L|K$ such that $\sigma\vartheta-\vartheta=j=
\sigma\vartheta'-\vartheta'$, then we have $\sigma(\vartheta-\vartheta')
=\vartheta-\vartheta'$. Since $\sigma$ is a generator of $\Gal L|K\isom
\Z/p\Z$, it follows that $\tau(\vartheta-\vartheta')=\vartheta-
\vartheta'$ for all $\tau\in\Gal L|K$, that is, $\vartheta-\vartheta'\in K$.

If $\vartheta'$ is another Artin-Schreier generator of $L|K$ and
$\sigma\vartheta'-\vartheta'=j'\in\F_p^\times$, then there is some $i\in\F_p^\times$
such that $ij=j'$. Since $\sigma i=i$, we then have that $\sigma i\vartheta-i\vartheta=i(\sigma\vartheta-\vartheta)
=ij=j'$. Then by what we have shown before, $\vartheta'=i\vartheta+c$ for some $c\in K$.

Conversely, if $\vartheta$ is an Artin-Schreier generator of $L|K$ and
if $i\in\F_p^\times$ and $c\in K$, then $(i\vartheta+c)^p-(i
\vartheta+c)=i(\vartheta^p-\vartheta)+c^p-c\in K$. But $i\vartheta+c$
cannot lie in $K$, so $K(i\vartheta+c)=L$ since $[L:K]$ is a prime. This
shows that also $i\vartheta+c$ is an Artin-Schreier generator of $L|K$.
\end{proof}

\begin{corollary}                           \label{ASinv}
Let $(L|K,v)$ be an Artin-Schreier extension of valued fields. Then
$v(\vartheta-K)$ is independent of the choice of the Artin-Schreier generator
$\vartheta$ of the extension, so it is an invariant of the extension.
\end{corollary}
\begin{proof}
Take two Artin-Schreier generators $\vartheta,\vartheta'$ of $L|K$. By Lemma~\ref{allASg} we can write
$\vartheta' =i\vartheta +c$ for some $i\in\F_p^\times$ and $c\in K$. Since $vi=0$ for every valuation, it follows
from assertion 5) of Proposition~\ref{propva-K} that $v(\vartheta'-K)=v(\vartheta-K)$.
\end{proof}

%

\begin{proposition}                             \label{propve-th}
Assume that $(K(\eta)|K,v)$ is a purely inseparable extension and $(K(\vartheta)|K,v)$ is an
Artin-Schreier extension of valued fields. Then $K(\eta,\vartheta)|K(\vartheta)$ is purely inseparable
and so the extension of $v$ from $K(\vartheta)$ to $K(\eta,\vartheta)$ is unique. If
\begin{equation}                                   \label{ve-th}
v(\eta-\vartheta)\> > \> v(\vartheta-K)\>,
\end{equation}
then the extension of $v$ from $K$ to $K(\vartheta)$ is unique. If in addition, $(K(\eta)|K,v)$ is immediate,
then $(K(\vartheta)|K,v)$ is immediate with defect~$p$.
\end{proposition}
\begin{proof}
Assume that (\ref{ve-th}) holds. Then $\vartheta\notin K^h$ since otherwise it would follow
from Theorem~2 of~\cite{Ku4}
that $K(\eta)|K$ is not purely inseparable, contradicting our assumption. Now our first assertion follows from
part 1) of Lemma~\ref{c1}. By part 6) of Proposition~\ref{propva-K}, (\ref{ve-th}) also implies
that $v(\vartheta-K)=v(\eta-K)$.
\pars
Now assume in addition that $(K(\eta)|K,v)$ is immediate. Then by part 1) of Proposition~\ref{propva-K},
$v(\vartheta-K)=v(\eta-K)$ has no maximal element.  Therefore, our second assertion follows from Lemma~\ref{ueGp1}.
\end{proof}

%
%
\subsection{Kummer extensions of prime degree of valued fields}                            \label{sectK}
Take a valued field $(K,v)$ of mixed characteristic, that is, $\chara K=0$ while $\chara Kv=p>0$. We consider
Kummer extensions of degree $p$. Such an extension is generated by an element $\eta$ such that $\eta^p\in K$.
If $(K(\eta)|K,v)$ is immediate, then it can be assumed that $\eta$ and hence also $\eta^p$ is a $1$-unit, i.e.,
$v(\eta-1)>0$. Indeed, since $(K(\eta)|K,v)$ is immediate, we have that $v\eta\in vK(\eta)=vK$,
so there is $c\in K$ such that $vc=-v\eta$. Then $v\eta c=0$, and since $\eta cv\in K(\eta)v=Kv$, there is
$d\in K$ such that $dv=(\eta cv)^{-1}$. Then $v(\eta cd)=0$ and $(\eta cd)v=1$. Hence $\eta cd$ is a $1$-unit.
Furthermore, $K(\eta cd)=K(\eta)$ and $(\eta cd)^p=\eta^pc^pd^p\in K$. Thus we can replace $\eta$ by $\eta cd$
and assume from the start that $\eta$ is a $1$-unit.

\pars
The next proposition follows from \cite[Corollary~3.6 and Proposition 3.7] {BKdrf}:
\begin{proposition}                                       \label{dist_approx}
Take a valued field of mixed characteristic and a Kummer defect extension as detailed above.
Then the distance $\dist(\eta,K)$ does not depend on the choice of the generator $\eta$
of the extension $(K(\eta)|K,v)$ as long as it is a $1$-unit and satisfies $\eta^p\in K$. Moreover,
\begin{equation}                                  \label{disteta1}
0\><\>\dist(\eta,K) \>\leq\> \left(\frac{vp}{p-1}\right)^-\>.
\end{equation}
\end{proposition}
Under the above conditions, the Kummer extension is immediate, hence we have that $v(\eta-K)\subseteq vK$. From the
definition of the distance it then follows that $v(\eta-K)$ is the intersection of the lower cut set of
$\dist(\eta,K)$ with $vK$, hence the distance determines uniquely the set $v(\eta-K)$, showing that this set does
not depend on the choice of the generator $\eta$. Moreover, from Lemma~\ref{translate} we obtain that 
inequality (\ref{Kde}) holds.

\parm

In \cite[Proposition 3.7]{BKdrf} we show that a Kummer extension $(K(\eta)|K,v)$, where $\eta$ is a $1$-unit with 
$\eta^p\in K$, has dependent defect if and only if
\[
\dist(a,K) \>=\> \frac{vp}{p-1} \,+\, H^-\>,
\]
for some proper convex subgroup $H$ of $\widetilde{vK}$. If $(K,v)$ has rank 1, then $\widetilde{vK}$ is 
archimedean ordered and therefore $H=\{0\}$. In this case, the above equation just becomes 
\[
\dist(\eta,K)\><\>\left(\frac{vp}{p-1}\right)^-\>.
\]
In view of Lemma~\ref{translate}, this translates to $v(\eta-K)<\!\!|\, \frac{vp}{p-1}$, as mentioned in the 
Introduction.

\bn
%
%
\subsection{Differences of $p$-th powers in mixed characteristic}                              \label{sectdpp}
\begin{lemma}                           \label{lvep-ap}
Take $\eta$ and $a$ to be two elements of some valued field of characteristic 0 with residue field characteristic $p>0$. If
\begin{equation}                                \label{ve-a}
v(\eta-a)\><\> \frac{vp}{p-1}+v\eta\>,
\end{equation}
then
\begin{equation}                           \label{vep-ap}
v(\eta^p-a^p) \>=\> pv(\eta-a)\>.
\end{equation}
\end{lemma}
\begin{proof}
Let $\zeta$ be a primitive $p$-th root of unity. We have that
\[
\eta^p-a^p \>=\>\prod_{i=1}^p(\zeta^i\eta-a)\>=\>\prod_{i=1}^p(\zeta^i\eta-\eta+\eta-a)\>.
\]
Using the well known fact that
\[
v(\zeta^i-1)\>=\>\frac{vp}{p-1}
\]
(see e.g.\ the proof of Lemma~2.9 of~\cite{Ku8}) and our assumption, we obtain that
\[
v(\zeta^i-1)+v\eta\>=\> \frac{vp}{p-1}+v\eta \> >\> v(\eta-a)\>.
\]
Hence we have that
\[
v(\zeta^i\eta-a)\>=\> v(\zeta^i\eta-\eta+\eta-a)\>=\>\min\{v(\zeta^i-1)+v\eta, v(\eta-a)\}\>=\>v(\eta-a)\>,
\]
which yields equation (\ref{vep-ap}).
\end{proof}

\bn
%
%
%
\section{The case of valued fields of characteristic $p>0$}                              \label{sectcharp}
%

%
%
\subsection{A basic transformation}                              \label{secttransf}
We describe a transformation that was introduced in \cite{Ku6}. Take a valued field $(K,v)$ of characteristic
$p>0$ and a (not necessarily immediate) purely inseparable extension $(K(\eta)|K,v)$ of degree $p$ of valued
fields such that $\eta^p\in K$ and that $v(\eta-K)$ is bounded from above by an element in $vK(\eta)$ (and hence
also by an element in $vK$). We are starting with the minimal polynomial $Y^p-\eta^p$ of $\eta$ over $K$ and turn
it into the separable polynomial
\begin{equation}                         \label{trans1}
Y^p-d^{p-1}Y-\eta^p\>,
\end{equation}
where $d\ne 0$. Setting $Y=dX$ and then dividing the resulting polynomial by $d^p$, we transform this polynomial
into the Artin-Schreier polynomial
\begin{equation}                         \label{trans2}
X^p-X-\frac{\eta^p}{d^p}\>.
\end{equation}
Under the condition that $vd$ is large enough, the following lemma describes the behaviour of the set $v(\eta-K)$
when $\eta$ is replaced by roots of these two polynomials.
\begin{lemma}                               \label{dXeqc}
Take $d\in K$ such that
\begin{equation}              \label{vd}
vd^{p-1}\> >\> pv(\eta-K) -v\eta\>.
\end{equation}
Let $\tilde\vartheta_d$ be a root of the polynomial (\ref{trans1}). Then
\[
\vartheta_d\>:=\> \frac{\tilde\vartheta_d}{d}
\]
is a root of (\ref{trans2}) and
$K(\vartheta_d)|K$ is an Artin-Schreier extension with a unique extension of $v$ from $K$ to
$K(\vartheta_d)$. Furthermore, we have that
\begin{equation}                            \label{dXeqceq1}
v(\tilde\vartheta_d-K)\>=\>v(\eta-K)\>,
\end{equation}
\begin{equation}                            \label{dXeqceq2}
v\left(\frac{\eta}{d}-\vartheta_d\right)\> >\> v(\vartheta_d-K)\>,
\end{equation}
and
\begin{equation}                   \label{vth-K}
v(\vartheta_d-K)\>=\> v(\eta-K)-vd\>.
\end{equation}
\end{lemma}
\begin{proof}
Once we prove that
\begin{equation}                            \label{dXeqceq3}
v(\eta-\tilde\vartheta_d)\> >\> v(\eta-K)\>,
\end{equation}
we obtain equation~(\ref{dXeqceq1}) by part 6) of Proposition~\ref{propva-K}, which in turn implies
equation~(\ref{vth-K}) by part 5) of Proposition~\ref{propva-K}. Further, again using part 5) of
Proposition~\ref{propva-K} again, we obtain
\[
v\left(\frac{\eta}{d}-\vartheta_d\right)\>=\>v(\eta-\tilde\vartheta_d)-vd\> >\> v(\eta-K)-vd\>=\>
v(\tilde\vartheta_d-K)-vd \>=\> v(\vartheta_d-K)\>,
\]
which proves (\ref{dXeqceq2}). Hence we will now prove (\ref{dXeqceq3}). We compute:
\begin{equation}          \label{e-th}
(\eta-\tilde\vartheta_d)^p\>=\> \eta^p-\tilde\vartheta_d^p \>=\> \eta^p - \eta^p - d^{p-1}\tilde\vartheta_d \>=\>
- d^{p-1}\tilde\vartheta_d \>.
\end{equation}

\pars
We set $Y=dX$ to obtain that $\vartheta_d$ is a root of the polynomial $d^pX^p-d^{p-1} dX-\eta^p$ and hence also
of the Artin-Schreier polynomial (\ref{trans2}).

\pars
Since $v\eta\in v(\eta-K)$, from (\ref{vd}) we obtain:
\[
(p-1)vd \>=\> vd^{p-1}\> >\> pv\eta - v\eta = (p-1)v\eta\>,
\]
so that
\[
v\frac{\eta^p}{d^p} \>=\> p(v\eta-vd)\><\>0\>.
\]
Hence we have that $v\vartheta_d<0$ and consequently,
$v\vartheta_d^p=pv\vartheta_d<v\vartheta_d$ and
\[
pv\frac{\eta}{d}\>=\>v\frac{\eta^p}{d^p}\>=\>\min\{v\vartheta_d^p,v\vartheta_d\}\>=\>v\vartheta_d^p
\>=\>pv\vartheta_d\>,
\]
which yields that
\[
v\tilde\vartheta_d\>=\> vd+v\vartheta_d \>=\> v\eta\>.
\]
From this together with (\ref{vd}) and (\ref{e-th}) we obtain:
\[
v(\eta-\tilde\vartheta_d)\>=\> \frac{1}{p}v(d^{p-1}\tilde\vartheta_d)\>=\> \frac{1}{p}(vd^{p-1}+v\eta)
\> >\> v(\eta-K)\>,
\]
as desired.

\pars
It remains to prove that $K(\vartheta_d)|K$ is nontrivial (and hence an Artin-Schreier extension),
and that the extension of $v$ from $K$ to $K(\vartheta_d)$ is unique. Since $v(\eta-K)$ is bounded from above in
$vK$ by assumption, the same holds for $v(\eta-K)-vd$ which by (\ref{vth-K})
is equal to $v(\vartheta_d-K)$. This implies in particular that
$\vartheta_d\notin K$, so the extension $K(\vartheta_d)|K$ is nontrivial. Since $\eta$ is purely
inseparable over $K$ by assumption, the same holds for $\eta/d$. Thus our assertion follows from (\ref{dXeqceq2})
together with Proposition~\ref{propve-th}, where we replace $\eta$ by $\eta/d$ and $\vartheta$ by $\vartheta_d\,$.
\end{proof}

\bn
%
%
%
\subsection{Proof of Theorems~\ref{MTfinAS} and~\ref{MTdepAS}}                  \label{sectcharppf}
Throughout, let $(K,v)$ be a valued field of characteristic $p>0$. First, we prove:
\sn
{\it if $K$ admits only finitely many distinct Artin-Schreier extensions, then $(K,v)$ is a semitame field,}
\sn
which is the assertion of Theorem~\ref{MTfinAS}.
Take a purely inseparable (not necessarily immediate) extension
$K(\eta)|K$ of degree $p$ such that $\eta^p\in K$ and that $v(\eta-K)$ is bounded from above by an element in
$vK$.

\begin{proposition}                    \label{infASdn}
Assume that $\alpha\in v(\eta-K)$ and $d\in K$ are such that $\alpha+vd> v(\eta-K)$ (that is,
(\ref{alph+vd}) holds for $\eta$ in place of $a$) and $vd^{p-1}>pv(\eta-K) -v\eta$ (that is, (\ref{vd}) holds).
Then with the notation from Lemma~\ref{dXeqc}, the Artin-Schreier extensions
\[
K(\vartheta_{d^n})|K\>, \;n\in\N\>,
\]
are pairwise distinct, and the extensions of $v$ from $K$ to $K(\vartheta_{d^n})$ are unique for all $n\in\N$.
\end{proposition}
\begin{proof}
It follows from Lemma~\ref{dXeqc} that for each $n\in\N$, $K(\vartheta_{d^n})|K$ is an Artin-Schreier extension
and the extension of $v$ from $K$ to all $K(\vartheta_{d^n})$ is unique.
From Lemma~\ref{distinct} we infer that the sets $v(\vartheta_{d^n}-K)=v(\eta-K)-vd^n$ are distinct.
Thus by Corollary~\ref{ASinv}, the extensions $K(\vartheta_{d^n})|K$ are distinct.
\end{proof}

\begin{lemma}                        \label{pidegp}
If the perfect hull of $K$ does not lie in the completion
of $(K,v)$, then $(K,v)$ admits a purely inseparable extension $K(\eta)|K$ of degree $p$ such that
$v(\eta-K)$ is bounded from above by an element in $vK$.
\end{lemma}
\begin{proof}
Assume that the perfect hull $K^{1/p^{\infty}}$ does not lie in the completion $K^c$ of $(K,v)$, and take an
element
$\tilde\eta\in K^{1/p^{\infty}} \setminus K^c$. We may assume that $\tilde\eta^p\in K^c$ (otherwise, we replace
$\tilde \eta$ by ${\tilde\eta}^{p^{\nu}}$ for a suitable $\nu\geq 1$). Since $\tilde\eta\notin K^c$, we have that
$v(\tilde\eta-K)$ is bounded from above by some $\alpha\in vK$ and $v(\tilde\eta^p-K^p)= pv(\tilde\eta-K)$ is
bounded from above by $p\alpha$. On the other hand, since $\tilde{\eta}^p\in K^c$, there
is some $b\in K$ such that $v(\tilde\eta^p-b)>p\alpha$. We choose $\eta\in K^{1/p}$ such that $\eta^p=b$.
Then $pv(\tilde\eta-\eta)=v(\tilde\eta^p-\eta^p)=v(\tilde\eta^p-b)>p\alpha\geq pv(\tilde\eta-K)$, which yields that
$v(\tilde\eta-\eta)>v(\tilde\eta-K)$. From this it follows by part 6) of Proposition~\ref{propva-K}
that $v(\eta-K)=v(\tilde\eta-K)$, showing that also $v(\eta-K)$ is bounded from above by $\alpha$.
\end{proof}

\pars
Since $v(\eta-K)$ is assumed to be bounded from above by an element in $vK$, there is some $d\in K$ which satisfies
$vd^{p-1}>pv(\eta-K) -v\eta$.
Then the same is true for every $d'\in K$ with $vd'\geq
vd$. Hence $\alpha\in v(\eta-K)$ and $d\in K$ can be chosen such that $\alpha+vd> v(\eta-K)$.
Thus we can use Lemma~\ref{pidegp} together with Proposition~\ref{infASdn} to obtain:
\begin{proposition}                         \label{propnotdense}
If the perfect hull of $K$ does not lie in the completion
of $(K,v)$, then $K$ admits infinitely many Artin-Schreier extensions.
\end{proposition}
\sn
By the equivalence of assertions a) and e) of Theorem~\ref{charst}, this proposition proves Theorem~\ref{MTfinAS}.

\parm
We will now prove the assertion of Theorem~\ref{MTdepAS}, which states:
\sn
{\it if $(K,v)$ admits an Artin-Schreier extension with dependent
defect, then it admits infinitely many distinct Artin-Schreier extensions with dependent defect.}
\sn
We assume that $(K,v)$ admits an Artin-Schreier extension with dependent defect, which means that it must be
obtained via the transformation described in Section~\ref{secttransf} from a purely inseparable defect extension
$(K(\eta)|K,v)$ of degree $p$. This is only possible if $v(\eta-K)$ is bounded from above by an element in
$vK$. Also, since the extension is of degree $p$ with nontrivial defect, the defect must be $p$ and the extension
must be immediate. Hence by part 1) of Proposition~\ref{propva-K},
$v(\eta-K)$ has no largest element and consequently, the same holds for the sets $v(\eta-K)-vd^n=
v(\vartheta_{d^n}-K)$. Therefore, since the extension of $v$ from $K$ to $K(\vartheta_{d^n})$ is unique by
Proposition~\ref{infASdn}, it follows from Lemma~\ref{ueGp1} that $(K(\vartheta_{d^n})|K,v)$ is immediate with
defect $p$. As this extension is obtained from a purely inseparable defect extension of degree $p$ by the
transformation described in Section~\ref{secttransf}, this defect is dependent by definition. Finally,
Proposition~\ref{infASdn} shows that the extensions $K(\vartheta_{d^n})|K$, $n\in\N$, are distinct. This
completes our proof.

\bn
%
%
\section{The case of valued fields of  mixed characteristic}                              \label{sectchar0}
Throughout this section, we assume that $(K,v)$ is a valued field of rank $1$ and of characteristic $0$ with
residue field of characteristic $p>0$. Further, we assume that $K$ contains a primitive $p$-th root of unity,
and that $(K(\eta)|K,v)$ is a Kummer extension, where $\eta$ is a $1$-unit with $\eta^p\in K$.

%
%
\subsection{A basic transformation in the mixed characteristic case}            \label{secttransf0}
Given $d\in K$, we transform the minimal polynomial $X^p-\eta^p$ of $\eta$ over $K$ into the polynomial
\begin{equation}                         \label{trans3}
f_{d} (X)\>:=\> X^p+h_d(X)-\eta^p
\end{equation}
with
\[
h_d(X)\>:=\> \sum_{i=1}^{p-1} \binom{p}{i} d^{p-i}X^i\>.
\]

\begin{lemma}                               \label{btrf0}
Assume that $vd<0$ and that
\begin{equation}                             \label{ve-K}
v(\eta-K)\><\>\frac{vp+(p-1)vd}{p}\>.    
\end{equation}
Take $\tilde\vartheta_d$ to be a root of the polynomial (\ref{trans3}). Then
\[
v(\eta-K)\>=\>v(\tilde\vartheta_d-K)\>.
\]
\end{lemma}
\begin{proof}
We compute the value of the coefficients of $h_d\,$, using our assumption (\ref{ve-K}):
\begin{equation}
v\left(\binom{p}{i} d^{p-i}\right) \>=\> vp+(p-i)vd\>\geq\> vp+(p-1)vd\> >\>pv(\eta-K)\>.
\end{equation}
Since $0=v(\eta-1)\in v(\eta-K)$, this shows that all coefficients of $h_d$ have positive value, while
$v\eta^p=0$; this forces $\tilde\vartheta_d$ to have value $0$ by the ultrametric triangle law. Consequently,
\begin{equation}                         \label{vtth>}
vh_d(\tilde\vartheta_d) \> >\> pv(\eta-K)\>.
\end{equation}

Suppose that there is $c\in K$ such that $v(\tilde\vartheta_d-\eta)\leq v(\eta-c)$. Combined with our assumption
(\ref{ve-K}), this implies that assumption (\ref{ve-a}) of Lemma~\ref{lvep-ap} holds for $\eta$ and for
$\tilde\vartheta_d$ in place of $a$ since $v\eta=0$. Hence by Lemma~\ref{lvep-ap} we have that
\[
v(\tilde\vartheta_d^p-\eta^p) \>=\> pv(\tilde\vartheta_d-\eta)\>.
\]
Using this together with (\ref{vtth>}) and the fact that
$\tilde\vartheta_d^p=\eta^p+h_d(\tilde\vartheta_d)$ by definition of $\tilde\vartheta_d$, we compute:
\begin{eqnarray*}
v(\tilde\vartheta_d-\eta) &=& \frac 1 p v(\tilde\vartheta_d^p-\eta^p) \>=\> \frac 1 p
v(\eta^p+h_d(\tilde\vartheta_d)-\eta^p)\\
&=& \frac 1 p vh_d(\tilde\vartheta_d)\> >\>v(\eta-K)\>,
\end{eqnarray*}
contradicting our assumption. This shows that $v(\tilde\vartheta_d-\eta)> v(\eta-c)$ for all $c\in K$, whence
$v(\tilde\vartheta_d-\eta)> v(\eta-K)$. Hence by part 6) of Proposition~\ref{propva-K},
$v(\eta-K)=v(\tilde\vartheta_d-K)$.
\end{proof}

\bn
%
%
%
\subsection{Proof of Theorem~\ref{MTmixed}}                              \label{sectchar0pf}
We let $(K,v)$ be a valued field of mixed characteristic and rank $1$ which contains a primitive $p$-th root of
unity. We have to prove the assertion of Theorem~\ref{MTmixed}, which states:
\sn
{\it If $(K,v)$ admits a Kummer extension of degree $p$ with super-dependent defect, then it admits infinitely
many distinct Kummer extensions of degree $p$ with super-dependent defect.}
\sn
We assume that $(K(\eta)|K,v)$ is a Kummer extension of degree $p$ with super-dependent defect, and that
$\eta$ is a $1$-unit with $\eta^p\in K$. Since $0=v\eta\in v(\eta-K)$ and $v(\eta-K)$ does not contain a
maximal element, it must contain positive elements. Since $(K(\eta)|K,v)$ is a super-dependent defect extension,
we know that (\ref{sd}) holds.
Since by assumption, $(K,v)$ is of rank $1$ and $v(\eta-K)$ does not contain a maximal element but is bounded,
$vK$ must be dense in $\widetilde{vK}$. Hence there is some $\td\in K$ such that $v\td<0$ and
\begin{equation}                             \label{ve-Kstrong}
v(\eta-K)\><\>\frac{vp}{p}+2v\td\>.
\end{equation}
Then inequality (\ref{ve-K}) of Lemma~\ref{btrf0} is satisfied, because
\[
\frac{vp}{p}+2v\td \><\>\frac{vp+(p-1)v\td}{p}\>.
\]
We obtain from Lemma~\ref{btrf0} that for a root $\tilde\vartheta_{\td}$ of the
polynomial $f_{\td}(X)$ defined in that lemma,
\[
v(\eta-K)\>=\>v(\tilde\vartheta_{\td}-K)\>.
\]

Now we set $X=\td Y$; dividing the resulting polynomial $f_{\td}(\td Y)$ by $\td^p$, we obtain the polynomial
\[
g_d(Y)\>:=\> Y^p + \sum_{i=2}^{p-1} \binom{p}{i} Y^i-\frac{\eta^p}{\td^p}\>.
\]
We observe that
\[
\vartheta_{\td}\>:=\> \frac{\tilde\vartheta_{\td}}{\td}
\]
is a root of $g_{\td}(Y)$ and that by part 5) of Proposition~\ref{propva-K},
\begin{equation}
v(\vartheta_{\td}-K)\>=\> v(\tilde\vartheta_{\td}-K) -v\td\>.
\end{equation}

Now we set $Y=X-1$, which turns the polynomial $g_{\td}(Y)$ into the polynomial
\[
X^p-\left(\frac{\eta^p}{\td^p} +1\right)
\]
with root
\[
\eta_{\td}\>:=\> \vartheta_{\td} +1\>.
\]
Again by part 5) of Proposition~\ref{propva-K}, we have that
\[
v(\eta_{\td}-K)\>=\>v(\vartheta_{\td}-K)\>=\> v(\tilde\vartheta_{\td}-K) -v\td\>=\> v(\eta-K) -v\td
\>\ne\> v(\eta-K)\>,
\]
where the last inequality holds since $v\td\ne 0$ and $v(\eta-K)$ is a bounded subset of an archimedean ordered abelian group. As $vd<0=v\eta$,
\[
\eta_{\td}^p\>=\> \frac{\eta^p}{\td^p} +1
\]
is a $1$-unit.

\pars
Since $vK$ is dense, there are infinitely many $\td'\in K$ with $v\td'<0$ that satisfy (\ref{ve-Kstrong})
in place of $\td$. With the same argument as above, we see that $v(\eta-K) -v\td\ne v(\eta-K) -v\td'$ if
$v\td\ne v\td'$. In this way we obtain infinitely many Kummer extensions $(K(\eta_{\td})|
K,v)$ with distinct sets $v(\eta_{\td}-K)$, which by Proposition~\ref{dist_approx}
shows that these extensions are pairwise distinct.

\pars
Since $(K(\eta)|K,v)$ is a nontrivial immediate extension, part 1) of Proposition~\ref{propva-K}
shows that the set $v(\eta-K)$ has no maximal element, while being bounded by assumption. The same is consequently
true for the sets $v(\eta-K)-v\td=v(\eta_{\td}-K)$. As the rank of $(K,v)$ is assumed to be $1$, part 2) of
Lemma~\ref{c2} shows that $(K(\eta_{\td})|K,v)$ is immediate and has defect $p$. Finally, this defect is
super-dependent since
\[
v(\eta_{\td}-K)\>=\> v(\eta-K)-v\td \><\>\frac{vp}{p}+2v\td-v\td\>=\>\frac{vp}{p}+v\td\>.
\]

\bn
%
%
\section{Filling a gap in \cite{Ku6}}                              \label{sectgap}
In \cite[Lemma 2.31]{Ku6} the following is stated:
\begin{lemma}                               \label{uniqextv}
Assume that $\chara K=p>0$ and $(K(\vartheta)|K,v)$ is an Artin-Schreier extension with Artin-Schreier generator
$\vartheta$. If $\dist (\vartheta,K)\leq 0^-$ and $v(\vartheta-K)$ has no maximal element, then the
extension of $v$ from $K$ to $K(\vartheta)$ is unique and $(K(\vartheta)|K,v)$ is immediate with defect $p$.
\end{lemma}
\n
In the proof it is written: ``In \cite{Ku4} we show that the assumption that $\dist (\vartheta,K)<0$ implies that
the extension of $v$ from $K$ to $K(\vartheta)$ is unique.'' Since $v(\vartheta-K)$ has no maximal element,
Lemma~\ref{ueGp1} shows that this assertion indeed implies that $(K(\vartheta)|K,v)$ is immediate. In addition,
$\dist (\vartheta,K)<\infty$ implies that $\vartheta\notin K$, so the extension is nontrivial and thus has defect
$p$.

However, the assertion was never proven in \cite{Ku4}. In order to complete the proof of Lemma~\ref{uniqextv}, we
will prove it here. As $v(\vartheta-K)$ has no maximal element, we know from part part 2) of
Proposition~\ref{propva-K} that $v(a-K)\subseteq vK\subseteq \widetilde{vK}$. Hence by Lemma~\ref{translate}
the assumption $\dist (\vartheta,K)\leq 0^-$ is equivalent to $v(\vartheta-K)<0$.

\pars
By part 1) of Lemma~\ref{c1}, it suffices to prove that $\vartheta\notin K^h$. Suppose
otherwise. Then by \cite[Theorem~1]{Ku4}, there are $\alpha\in vK$ and a convex subgroup $H\ne\{0\}$ of $vK$
such that the coset $\alpha+H$ is cofinal in $v(\vartheta-K)$. Since $v(\vartheta-K)<0$, we must have that
 $\alpha+H<0$, which yields that $-\alpha>H$ and $-\alpha>0$. As $\alpha\in\alpha+H$, there is some $c\in K$
such that $v(\vartheta-c)\geq\alpha$. Since $\vartheta-c$ is also an Artin-Schreier generator of the extension and
$v(\vartheta-c-K)=v(\vartheta-K)$, we may assume w.l.o.g.\ that $v\vartheta\geq\alpha$.

Let $X^p-X-b$ with $b\in K$ be the minimal polynomial of $\vartheta$ over $K$, and take $d\in K$ with $vd=-\alpha$.
Then by part 5) of Proposition~\ref{propva-K}, $H$ is cofinal in $v(d\vartheta-K)$. Further,
$d\vartheta$ is a root of $d^{-p}X^p-d^{-1}X-b$ and hence also of
\[
X^p-d^{p-1}X-d^pb\>.
\]
Let $\eta$ be the root of $X^p-d^p b$. We wish to show that
\[
v(d\vartheta-K)\>=\>v(\eta-K)\>.
\]
We compute:
\begin{eqnarray*}
pv(d\vartheta-\eta) &=& v((d\vartheta)^p-\eta^p)\>=\>v(d^p\vartheta+d^pb-d^pb)\>=\>vd^p\vartheta\\
&\geq& -p\alpha  +\alpha\>\geq\> -\alpha \> >\> H\>.
\end{eqnarray*}
Since $H$ is a convex subgroup of $vK$, this implies that
$v(d\vartheta-\eta) \>\geq\> -\frac 1 p \alpha \> >\> H\>$, and as $H$ is cofinal in $v(d\vartheta-K)$, we obtain:
\[
v(d\vartheta-\eta) \> >\> v(d\vartheta-K)\>.
\]
By construction, $K(\eta)|K$ is purely inseparable, hence by Theorem~2 of~\cite{Ku4},
$d\vartheta$ cannot lie in the henselization $K^h$. We have thus shown that $\vartheta\notin K^h$, as desired.

\end{document}